\documentclass[11pt]{article}

\usepackage[english]{babel}
\usepackage[a4paper,top=3cm,bottom=3cm,left=2cm,right=2cm]{geometry}
\usepackage[utf8]{inputenc}
\usepackage[T1]{fontenc}
\usepackage[all,cmtip]{xy}				
\usepackage{amsmath,amssymb,amsthm}	
\usepackage[mathscr]{eucal}			
\usepackage{amsmath,amssymb}
\usepackage{amsthm}
\usepackage{amsfonts}
\usepackage{enumerate}
\usepackage{mathabx}
\usepackage[all,cmtip]{xy}
\usepackage{amscd}

\usepackage[svgnames]{xcolor}
\usepackage{listings}

\usepackage{enumitem}

\theoremstyle{plain} 
\newtheorem{thm}{Theorem}[section]
 
\newtheorem{lemma}[thm]{Lemma} 
\newtheorem{prop}[thm]{Proposition} 

\newtheorem{conj}[thm]{Conjecture}

\theoremstyle{definition} 

\newtheorem{rmk}[thm]{Remark} 
\newtheorem{ex}[thm]{Example}
\newtheorem*{ack}{Acknowledgements}
\newtheorem*{plan}{Plan of the paper}
\newtheorem*{notation}{Notation}
\newtheorem*{note}{Note}
\newtheorem*{rw}{Related works}
		
\DeclareMathOperator{\C}{\mathbb{C}}
\DeclareMathOperator{\p}{\mathbb{P}}
\DeclareMathOperator{\Q}{\mathbb{Q}}		
\DeclareMathOperator{\Z}{\mathbb{Z}}
\DeclareMathOperator{\R}{\mathbb{R}}

\DeclareMathOperator{\Hom}{\text{Hom}}

\DeclareMathOperator{\A}{\mathcal{A}}
\DeclareMathOperator{\D}{\text{D}^b}
\DeclareMathOperator{\FM}{\text{FM}}
	
\SetLabelAlign{parright}{\parbox[t]{\labelwidth}{\raggedleft#1}}

\setlist[description]{style=multiline,topsep=10pt,leftmargin=1cm,font=\normalfont,%
    align=parright}	
	
\title{\textbf{Fourier-Mukai partners for very general special cubic fourfolds}}
\author{Laura Pertusi}
\date{}

\begin{document}
\pagestyle{plain}
\maketitle

\begin{abstract} 
We exhibit explicit examples of very general special cubic fourfolds with discriminant $d$ admitting an associated (twisted) K3 surface, which have non-isomorphic Fourier-Mukai partners. In particular, in the untwisted setting, we show that the number of Fourier-Mukai partners for a very general special cubic fourfold with discriminant $d$ and having an associated K3 surface, is equal to the number $m$ of Fourier-Mukai partners of its associated K3 surface, if $d \equiv 2 (\text{mod}\,6)$; else, if $d \equiv 0 (\text{mod}\,6)$, the cubic fourfold has $\lceil m/2 \rceil$ Fourier-Mukai partners.
\end{abstract}

\section{Introduction}

A cubic fourfold $Y$ is a (smooth) hypersurface of degree $3$ in $\p^5_{\C}$. In \cite{Kuz1}, Kuznetsov studied the derived category $\D(Y)$ of bounded complexes of coherent sheaves on $Y$ to address the problem of the (non)rationality of the cubic fourfold. More precisely, for $i=0,1,2$, let $\mathcal{O}_Y(i)$ be the pullback of the line bundle $\mathcal{O}_{\p^5_{\C}}(i)$ with respect to the embedding of $Y$ in $\p^5_{\C}$. The derived category $\D(Y)$ admits a semiorthogonal decomposition of the form 
$$\D(Y)=\langle \A_Y, \mathcal{O}_Y, \mathcal{O}_Y(1), \mathcal{O}_Y(2) \rangle,$$
where $\A_Y$ is the right orthogonal of the subcategory of $\D(Y)$ generated by $\lbrace \mathcal{O}_Y, \mathcal{O}_Y(1), \mathcal{O}_Y(2) \rbrace$, i.e. 
$$\A_Y:=\lbrace F \in \D(Y): \Hom_{\D(Y)}(\mathcal{O}_Y(i),F)=0 \text{ for every } i=0,1,2 \rbrace.$$
Moreover, the triangulated subcategory $\A_Y$ has certain similarities with the bounded derived category of coherent sheaves on a K3 surface. Indeed, the Serre functor on $\mathcal{A}_Y$ is the shift by two and the Hochschild cohomology of $\mathcal{A}_Y$ is isomorphic to that of a K3 surface (see \cite{Kuz3}, Corollary 4.3 and \cite{Kuz2}, Proposition 4.1). 

The only non-trivial piece $\A_Y$ of the decomposition above should carry the information about the birational type of the cubic hypersurface. In fact, it has been conjectured that the cubic fourfold $Y$ is rational if and only if the category $\A_Y$ is equivalent to the derived category of coherent sheaves on a K3 surface (see \cite{Kuz1}, Conjecture 1.1). To support this guess, Kuznetsov proved in \cite{Kuz1} that the cubic fourfolds which were known to be rational satisfy this condition (see also \cite{Ricc}). Furthermore, this conjecture descends from a more general one, concerning the Clemens-Griffiths components associated to a (maximal) semiorthogonal decomposition (see \cite{Kuz2}, Section 3). 
 
On the level of the Hodge theory, the existence of an associated K3 surface as an indicator of rationality was deeply studied (see \cite{Hass2}, for a complete survey). Actually, Kuznetsov's conjecture would imply that a cubic fourfold with a Hodge associated K3 surface is rational, by results of Addington, Thomas and Bayer, Lahoz, Macrì, Nuer, Perry and Stellari, relating the categorical and the Hodge theoretical setting (see \cite{AT}, Theorem 1.1 and \cite{BLM+}, Corollary 1.7). Nevertheless, these conjectures have not been proved yet. 

In \cite{Huy}, Huybrechts deeply studied the category $\A_Y$, in order to develop a theory for cubic fourfolds which parallels that of the derived category of a (twisted) K3 surface. In particular, he proved the analogous version for $\A_Y$ of some results concerning Fourier-Mukai partners of a K3 surface. In general, we recall that a functor $\Xi: \D(Z) \rightarrow \D(Z')$, between the derived categories of two algebraic varieties $Z$ and $Z'$, is of \emph{Fourier-Mukai type} if there exist an object $K$ in the derived category $\D(Z \times Z')$ of the product and an isomorphism of exact functors
\[
\Xi(-) \cong \Phi_K(-):= Rp_{Z' *}(K \overset{\text{L}}{\otimes} Lp_{Z}^*(-)),
\] 
where $p_Z: Z \times Z' \rightarrow Z$ and $p_{Z'}: Z \times Z' \rightarrow Z'$ are the natural projections (see \cite{Kuz2}, Section 1.5). A cubic fourfold $Y'$ is a \textbf{Fourier-Mukai partner} of $Y$ if there exists an equivalence of categories 
$$\A_Y \xrightarrow{\sim} \A_{Y'}$$
which is of Fourier-Mukai type, i.e.\ such that the composition
$$\D(Y) \xrightarrow{i^*} \A_Y \xrightarrow{\sim} \A_{Y'} \hookrightarrow \D(Y')$$
is a Fourier-Mukai functor; here, $i^*$ denotes the left adjoint functor of the full inclusion $i: \A_Y \hookrightarrow\D(Y)$. Usually in the literature this denomination is used to identify smooth projective varieties with equivalent derived categories. However, cubic fourfolds satisfying this condition are isomorphic by a result of Bondal and Orlov (see \cite{BO}, Theorem 3.1). Thus, it becomes interesting to address the same problem by considering the K3 categories.   

Huybrechts showed that the number of (isomorphism classes of) Fourier-Mukai partners for a cubic fourfold $Y$ is finite (see \cite{Huy}, Theorem 1.1), as in the case of Fourier-Mukai partners for a K3 surface (see \cite{BM}, Proposition 5.3). Moreover, he proved that a very general cubic fourfold $Y$, i.e.\ such that $\text{rk}(H^{2,2}(Y,\Z))=1$, has no non-trivial Fourier-Mukai partners (see \cite{Huy}, Corollary 3.6).

It is natural to ask whether a \emph{special} cubic fourfold $Y$, i.e.\ such that $\text{rk}(H^{2,2}(Y,\Z)) \geq 2$, admits Fourier-Mukai partners which are not isomorphic to $Y$. In particular, we may wonder if for special cubic fourfolds it is possible to prove a version of Theorem 1.7 and Corollary 1.8 of \cite{Og},  which state that there are examples of K3 surfaces having a prescribed number of non-isomorphic Fourier-Mukai partners.

In the third section of this paper, we prove that the answer is positive in the case that the rank of $H^{2,2}(Y,\Z)$ is exactly two and the cubic fourfold $Y$ admits an associated K3 surface $X$ with ``enough'' non-trivial Fourier-Mukai partners. More precisely, given a positive integer $d$, we denote by $\mathcal{C}_d$ the divisor parametrizing special cubic fourfolds with discriminant $d$ (see Section 2.1). We recall that:
\begin{itemize}
\item (see \cite{Hass}, Theorem 1.0.1) the divisor $\mathcal{C}_d$ is non empty if and only if 
\begin{equation}\tag{0}
\label{0}
d>6 \text{ and } d\equiv 0,2 (\text{mod}\,6);
\end{equation}
\item (see \cite{Hass}, Theorem 1.0.2 or Section 2.1) a cubic fourfold $Y \in \mathcal{C}_d$ has an \emph{associated K3 surface} if and only if 
\begin{equation}\tag{\textbf{a}}
4 \nmid d,\: 9 \nmid d, \: p \nmid d \: \text{ for any odd  prime } p \text{ such that } p \equiv 2 (\text{mod}\,3).  
\end{equation}
\end{itemize}

The first result of this paper is a counting formula for the number of Fourier-Mukai partners for very general special cubic fourfolds admitting an associated K3 surface.

\begin{thm}
\label{thmmio}
Let $d$ be a positive integer satisfying $\emph{(0)}$ and $\emph{\textbf{(a)}}$. Let $Y$ be a very general special cubic fourfold in $\mathcal{C}_d$ and let $m$ be the number of non-isomorphic Fourier-Mukai partners of an associated K3 surface to $Y$. Then, the cubic fourfold $Y$ has exactly $m$ non-isomorphic Fourier-Mukai partners, when $d \equiv 2 (\emph{mod}\,6)$; otherwise, if $d \equiv 0 (\emph{mod}\,6)$, the number of non-isomorphic Fourier-Mukai partners of $Y$ is equal to $\lceil m/2 \rceil$.
\end{thm}

As a consequence of Theorem \ref{thmmio}, we deduce that there exist cubic fourfolds admitting an arbitrary number of Fourier-Mukai partners, depending on the number of distinct odd primes in the prime factorization of the discriminant (see Proposition \ref{propmia}).

More generally, we recall that a cubic fourfold $Y \in \mathcal{C}_d$ has an \emph{associated twisted K3 surface} (see \cite{Huy}, Section 2.4 or Section 2.4) if and only if
\begin{equation}\tag{\textbf{a'}}
n_i \equiv 0 (\text{mod}\,2) \text{ for all primes } p_i \equiv 2 (\text{mod}\,3) \text{ in } 2d=\prod p_i^{n_i}.
\end{equation} 
A weaker formulation of Theorem \ref{thmmio} holds for very general special cubic fourfolds $Y$ in $\mathcal{C}_d$, admitting an associated twisted K3 surface $(X,\alpha)$, if $9$ does not divide the discriminant $d$. Indeed, in Section 4.1, we show that the number of non-isomorphic twisted Fourier-Mukai partners of $(X,\alpha)$ with the Brauer class of the same order as $\alpha$, gives a lower bound for the number of Fourier-Mukai partners of the cubic fourfold. 
 
\begin{thm}
\label{propmiatwist}
Let $d$ be a positive integer satisfying $\emph{(0)}$ and $\emph{\textbf{(a')}}$. Assume that $9$ does not divide $d$. Let $Y$ be a very general special cubic fourfold in $\mathcal{C}_d$ with associated twisted K3 surface $(X,\alpha)$, where $\alpha$ has order $\kappa$; let $m'$ be the number of non-isomorphic Fourier-Mukai partners of $(X,\alpha)$ with Brauer class of order $\kappa$. Then the cubic fourfold $Y$ admits at least $m'$ non-isomorphic Fourier-Mukai partners, when $d \equiv 2 (\emph{mod}\,6)$; otherwise, if $d \equiv 0 (\emph{mod}\,6)$, the number of non-isomorphic Fourier-Mukai partners of $Y$ is at least $\lceil m'/2 \rceil$.  
\end{thm}

In particular, under the hypotheses of Theorem \ref{propmiatwist}, the value $m'$ is controlled by the number of distinct primes in the prime factorization of $d/2$ divided by the square of the order of the Brauer class $\alpha$ and by the Euler function evaluated in $\text{ord}(\alpha)$, as we show in Proposition \ref{countm'}.

These results complete the expected analogy between cubic fourfolds and K3 surfaces, stated in \cite{Huy}, and prove that the isomorphism class of a cubic fourfold is in general not determined by its K3 subcategory. They also represent a first step in order to understand whether cubic fourfolds which are Fourier-Mukai partners are birational.

\begin{rw}
In \cite{BMMS}, Proposition 6.3, the authors prove that a cubic fourfold $Y$ in $\mathcal{C}_8$ with $H^{2,2}(Y,\Z)$ of rank $2$ has only one isomorphism class of Fourier-Mukai partners. Also in \cite{BMMS}, Remark 6.4, they explain how to construct examples of non-isomorphic Fourier-Mukai partners for cubic fourfolds in $\mathcal{C}_8$ with $H^{2,2}(Y,\Z)$ of rank $> 2$.
\end{rw} 

\begin{plan}
In Section 2 we recall some preliminary material on cubic fourfolds and  K3 surfaces we will use in the next sections. In Section 3.1 we count the period points defined by the Fourier-Mukai partners of a K3 surface associated to a very general special cubic fourfold. Section 3.2 is devoted to the proof of Theorem \ref{thmmio}. In Section 4.1 we prove Theorem \ref{propmiatwist}. In Section 4.2 we give a counting formula for twisted non-isomorphic Fourier-Mukai partners of a twisted K3 surface which is associated to a cubic fourfold as in Theorem \ref{propmiatwist}. In Section 4.3 we explicit the lower bound of Theorem \ref{propmiatwist}.
\end{plan}

\begin{notation}
We use the following terminology: a cubic fourfold $Y$ is very general if $\text{rk}(H^{2,2}(Y,\Z))=1$, while a very general special cubic fourfold $Y$  has $\text{rk}(H^{2,2}(Y,\Z))=2$ and belongs to the complement of a countable union of divisors in a divisor $\mathcal{C}_d$.

The number of non-isomorphic Fourier-Mukai partners of a K3 surface is denoted by $m$, which could be also equal to $1$.
\end{notation}

\begin{note}
Using the new results in \cite{BLM+}, it is possible to prove that Theorem \ref{thmmio} and Theorem \ref{propmiatwist} hold for every cubic fourfold with associated (twisted) K3 surface and $\text{rk}(H^{2,2}(Y,\Z))=2$, not only for the very general special ones. See Remark \ref{rmk_BLMS+Huy}, Remark \ref{rmk_extwithBLM+} and Remark \ref{rmk_extwithBLM+twist} for more details.
\end{note}

\begin{ack}
I am grateful to Daniel Huybrechts who read a preliminary version of this paper, for his valuable comments and for answering a question concerning Theorem 1.5 of \cite{Huy}. It is a pleasure to thank my advisor Paolo Stellari who proposed me the topic, for his useful suggestions and helpful discussions. Finally, I want to thank Riccardo Moschetti for his interesting observations and the referee for careful reading of the paper and for expert suggestions. 
\end{ack} 
 
\section{Recollection of results and notation} 

In this section we recall some results about cubic fourfolds and K3 surfaces we will use in the next. 

\subsection{Special cubic fourfolds and associated K3 surface}
Let $Y$ be a cubic fourfold; we denote by $H^4(Y,\Z)(1)$ the degree four integral cohomology group of $Y$ with weight two Hodge structure and intersection form $(\,,\,)$ with reversed sign. By classical results of Hodge theory and classification of quadratic forms, we have that $H^4(Y,\Z)(1)$ is isometric to the odd unimodular lattice $L:=\Z^{\oplus 2} \oplus \Z(-1)^{\oplus 21}$. It contains an element $h$ such that $(h,h)=-3$, corresponding to the square of the class of a hyperplane in $Y$. Moreover, the primitive lattice $H^4(Y,\Z)_{\text{prim}}(1)$ (with weight two Hodge structure and intersection form with reversed sign) is isometric to 
$$L^0:= A_2(-1) \oplus U^{\oplus 2} \oplus E_8(-1)^{\oplus 2};$$
here $U$ denotes the free group $\Z^{\oplus 2}$ with intersection form $\begin{pmatrix}
0 & 1 \\ 
1 & 0
\end{pmatrix}$ 
, $E_8(-1)$ is the unique even, unimodular lattice of signature $(0,8)$ and 
$A_2(-1):=\begin{pmatrix}
-2 &  1 \\ 
1  & -2
\end{pmatrix}$
(see \cite{Hass}, Proposition 2.1.2). We set
\begin{equation}
\label{quadric}
Q:=\lbrace y \in \p(L^0 \otimes \C): (y,y)=0, (y,\bar{y})>0 \rbrace.
\end{equation}
The choice of a connected component $\mathcal{D}^{\prime}$ of $Q$ determines the local period domain for cubic fourfolds. Let $\Gamma^{+}$ be the subgroup of the group of automorphism of $L$, preserving the class $h$ and the component $\mathcal{D}^{\prime}$. The \emph{global period domain} of cubic fourfolds is the quotient $\mathcal{D}:= \Gamma^+ \setminus \mathcal{D}^{\prime}$. We denote by $\mathcal{C}$ the moduli space of cubic fourfolds and let
$$\tau: \mathcal{C} \rightarrow \mathcal{D}$$
be the \emph{period map}. Voisin proved that $\tau$ is an open immersion, i.e.\ Torelli Theorem holds for cubic fourfolds (see \cite{Voi}). 

A cubic fourfold $Y$ is \emph{special} if there exists a rank-two (negative definite) primitive sublattice $(K, (\,,\,))$ of $H^4(Y,\Z) \cap H^{2,2}(Y)$, containing the class $h$. This lattice $K$ is a \emph{labelling} for $Y$ and the discriminant of the pair $(Y,K)$ is the determinant of the intersection matrix of $K$. We will write $K_d$ to underline the fact that the labelling has discriminant $d$. By Hassett's work, special Hodge structures with labelling $K_d$ form a divisor $\mathcal{D}_d'$ in the local period domain. If $\mathcal{D}_d=\Gamma^+ \setminus \mathcal{D}_d^{\prime}$, then $\mathcal{C}_d=\mathcal{C}\cap \mathcal{D}_d$ (via $\tau$) is the irreducible divisor in $\mathcal{C}$ of \emph{special cubic fourfolds of discriminant} $d$. By \cite{Hass}, Theorem 1.0.1, the divisor $\mathcal{C}_d$ is non empty if and only if 
\begin{equation}\tag{0}
d>6 \text{ and } d\equiv 0,2 (\text{mod}\,6).
\end{equation}
Hassett also gave numerical conditions on $d$ which ensure the existence of an \emph{associated K3 surface}.

\begin{thm}[\cite{Hass}, Theorem 1.0.2]
\label{propHass}
Let $Y$ be a cubic fourfold in $\mathcal{C}_d$ with labelling $K_d$. There exist a K3 surface $X$ with polarization class of degree $d$ and an isometry of Hodge structures 
$$K_d^{\perp} \cong H^2(X,\Z)_{\text{prim}}$$
between the orthogonal sublattice to $K_d$ in $H^4(Y,\Z)(1)$ and the degree two primitive cohomology of the K3 surface, if and only if $d$ satisfies the following condition:
\begin{equation}\tag{\textbf{a}}
4 \nmid d,\: 9 \nmid d, \: p \nmid d \: \text{ for any odd  prime } p \text{ such that } p \equiv 2 (\emph{mod}\,3).  
\end{equation} 
\end{thm}

We point out the following property concerning the discriminant group $d(K_d^{\perp})$ of $K_d^{\perp}$, endowed with the discriminant form $q_{K_d^{\perp}}$ induced by the intersection form.

\begin{prop}[\cite{Hass}, Proposition 3.2.6]
\label{K_dperp}
If $d \equiv 0 (\emph{mod}\,6)$, then $d(K_d^{\perp}) \cong \Z / \frac{d}{3} \mathbb{Z} \oplus \Z / 3 \mathbb{Z}$, which is cyclic unless nine divides $d$. Furthermore, we may choose this isomorphism so that 
$$q_{K_d^{\perp}}((0,1))\equiv -\frac{2}{3}(\emph{mod}\,2\mathbb{Z}) \quad \text{and} \quad q_{K_d^{\perp}}((1,0))\equiv \frac{3}{d}(\emph{mod}\,2\mathbb{Z}).$$ 
If $d \equiv 2 (\emph{mod}\,6)$, then $d(K_d^{\perp}) \cong \Z / d \mathbb{Z}$. Furthermore, we may choose a generator $g$ so that 
$$q_{K_d^{\perp}}(g)\equiv \frac{1-2d}{3d}(\emph{mod}\,2\mathbb{Z}).$$
\end{prop}

\subsection{Immersion into the moduli spaces of K3 surfaces} 
In \cite{Hass}, Section 5.3, Hassett proved that the existence of an isometry of Hodge structures as in Theorem \ref{propHass} allows an identification between the moduli space of marked special cubic fourfolds of discriminant $d$ and the moduli space of degree $d$ polarized K3 surfaces. Let us explain this observation. We fix a rank-two, negative definite, primitive sublattice $K_d \subset L$ of discriminant $d$, containing $h$. We write $\Gamma^+_d$ to denote the subgroup of the group of automorphisms of $L$ fixing the class $h$ and preserving the labelling $K_d$. Let $\mathcal{D}_d^{\text{lab}}$ be the period domain which parametrizes Hodge structures $x \in \mathcal{D}^{\prime}$ with $K_d \subset H^{2,2}(x) \cap L$, modulo the action of $\Gamma_d^+$, i.e.\
$$\mathcal{D}_d^{\text{lab}}:= \Gamma_d^+ \setminus \mathcal{D}_d^{\prime}.$$
We say that $\mathcal{D}_d^{\text{lab}}$ is the period domain of \emph{labelled special Hodge structures with discriminant} $d$. Notice that $\mathcal{D}_d^{\text{lab}}$ is birational to $\mathcal{D}_d$ via the morphism $\mathcal{D}_d^{\text{lab}} \rightarrow \mathcal{D}$. Actually, a point in $\mathcal{D}_d$ with the minimal number of algebraic classes has a unique labelling. In particular, $\mathcal{D}_d^{\text{lab}}$ is the normalization of $\mathcal{D}_d$ (see \cite{Hass}, Section 3.1)

Let, now, $G_d^+$ be the subgroup of $\Gamma_d^+$ of automorphisms acting trivially on $K_d$. Then, Hassett defines the period domain of \emph{marked special Hodge structures of discriminant $d$} as the quotient
$$\mathcal{D}_d^{\text{mar}}:= G_d^+ \setminus \mathcal{D}_d^{\prime}.$$
In this new space, two cubic fourfolds having the same labelling $K_d$ which comes from different primitive embeddings in $ H^{2,2}(-) \cap L$ are not identified. The relation between $\mathcal{D}_d^{\text{mar}}$ and $\mathcal{D}_d^{\text{lab}}$ is explained in the following proposition.

\begin{prop}[\cite{Hass}, Proposition 5.3.1]
\label{markedvslab}
The group $G_d^+$ is equal to $\Gamma_d^+$ (resp.\ the group $G_d^+$ is an index-two subgroup of $\Gamma_d^+$), if $d \equiv 2 (\emph{mod}\,6)$ (resp.\ if $d \equiv 0 (\emph{mod}\,6)$). 

The forgetful map $\rho: \mathcal{D}_d^{\text{mar}} \rightarrow \mathcal{D}_d^{\text{lab}}$ is an isomorphism (resp.\ a double cover), if $d \equiv 2 (\emph{mod}\,6)$ (resp.\ if $d \equiv 0 (\emph{mod}\,6)$). 
\end{prop}

\begin{rmk}
\label{rmkgamma}
Assume $d \equiv 0 (\text{mod}\,6)$ and let $y$ be a point in $\mathcal{D}_d^{\text{lab}}$. By Proposition \ref{markedvslab}, the fiber on $y$ of the map $\mathcal{D}_d^{\text{mar}} \rightarrow \mathcal{D}_d^{\text{lab}}$ contains two elements, which we denote by $y_1$ and $y_2$. Then, the Hodge structures on $L^0$ represented by $y_1$ and $y_2$ are related by the automorphism $\gamma$ of $L^0$, acting as the multiplication by $-1$ on the two hyperbolic planes and as the identity elsewhere. The automorphism $\gamma$ is an element of $\Gamma^+$, but it does not belong to $G_d^+$ for these values of the discriminant (see \cite{Hass}, the proof of Proposition 5.3.1).  
\end{rmk}

On the other hand, let us recall the construction of the global period domain for \emph{degree $d$ polarized K3 surfaces}. The degree two cohomology group of a K3 surface with the usual intersection form $(\,,\,)$ is isometric to the lattice $\Lambda:=E_8(-1)^{\oplus 2} \oplus U^{\oplus 3}$. We denote by $\Sigma_d$ the group of automorphisms of $\Lambda$, preserving the element $l:=e_1 + (d/2)f_1 \in U$. Let $\Lambda_d^0$ be the orthogonal complement of $l$ in $\Lambda$. Then, the quadric
$$\lbrace x \in \p(\Lambda^0_d \otimes \C): (x,x)=0, (x,\bar{x})>0 \rbrace$$
has two connected components $\mathcal{N}_d^{\prime}$ and $\overline{\mathcal{N}}_d^{\prime}$. Let $\Sigma_d^+$ be the subgroup of $\Sigma_d$ preserving the component $\mathcal{N}_d^{\prime}$: the quotient 
$$\mathcal{N}_d:=\Sigma_d^+ \setminus \mathcal{N}_d^{\prime}$$
is the global period domain for K3 surfaces with polarization class of degree $d$.

The following result follows from the fact that an isometry $K_d^\perp \cong \Lambda_d^0$ induces the isomorphisms $\mathcal{D}_d' \cong \mathcal{N}_d'$ and $G_d^+ \cong \Sigma_d^+$.

\begin{thm}[\cite{Hass}, Theorem 5.3.2, 5.3.3]
\label{Immmodspace}
Let $d$ be a positive integer satisfying conditions \emph{(0)} and \emph{\textbf{(a)}}. Then, there exists an isomorphism 
$$j_d: \mathcal{D}_d^{\text{mar}} \rightarrow \mathcal{N}_d,$$
induced by $K_d^\perp \cong \Lambda_d^0$, which is unique up to the choice of an element in $\emph{Iso}(d(K_d^{\perp}),d(\Lambda_d^0))/(\pm 1)$. 
\end{thm}

\subsection{Mukai lattice for $\A_Y$}
Let us now consider the categorical framework. The subcategory $\A_Y$ of a cubic fourfold $Y$ behaves in a certain way as the derived category of a K3 surface. Indeed, the Serre functor on $\mathcal{A}_Y$ is the shift by two and the Hochschild cohomology of $\mathcal{A}_Y$ is isomorphic to that of a K3 surface (see \cite{Kuz3}, Corollary 4.3 and \cite{Kuz2}, Proposition 4.1). In \cite{Kuz1}, Kuznetsov proved that for certain special cubic fourfolds $Y$, there exist a K3 surface $X$ and an equivalence of categories $\A_Y \xrightarrow{\sim} \D(X)$. In general, if this condition is satisfied, we say that $\A_Y$ is \emph{geometric}. In \cite{AT}, Addington and Thomas explained the relation between Kuznetsov's K3 surface and Hassett's Hodge theoretic associated K3 surface. Indeed, let $K_{\text{top}}(Y)$ be the topological K-theory of $Y$, which in this case is simply the Grothendieck group of topological complex vector bundles on $Y$, endowed with the Euler pairing $\chi$. Let
$$v_M: K_{\text{top}}(Y) \otimes \Q \xrightarrow{\cong} \bigoplus_{p=0}^4 H^{2p}(Y,\Q)(p),$$ 
be the isomorphism induced by the Mukai vector, defined as $v_M(-):=\text{ch}(-)\sqrt{\text{td}(Y)}$. 
Then $v_M$ induces a weight-zero Hodge structure on
$$K_{\text{top}}(\A_Y):=\lbrace \mathsf{k} \in K_{\text{top}}(Y): \chi([\mathcal{O}_Y(i)],\mathsf{k})=0, \text{ for all } i=0,1,2 \rbrace.$$
We denote by $\tilde{H}(\A_Y,\Z)$ the lattice $K_{\text{top}}(\A_Y)(-1)$ with the induced weight-two Hodge structure and Euler form with reversed sign: it is isomorphic to the lattice $\tilde{\Lambda}:=E_8(-1)^{\oplus 2} \oplus U^{\oplus 4}$ and it is called the \emph{Mukai lattice} of $Y$ (see \cite{AT}, Section 2.3). Let
$$N(\A_Y):= \tilde{H}^{1,1}(\A_Y,\Z)=\tilde{H}^{1,1}(\A_Y) \cap \tilde{H}(\A_Y,\Z)$$
be the \emph{generalized Néron-Severi lattice} of $\A_Y$ and we denote by $T(\A_Y)$ its orthogonal complement in $\tilde{H}(\A_Y,\Z)$, which is the \emph{generalized trascendental lattice} of $\A_Y$. Then, there exist two elements $\lambda_1, \lambda_2$ in $N(\A_Y)$ spanning a rank two sublattice with intersection matrix
\[
A_2:=\begin{pmatrix}
2  & -1 \\ 
-1 & 2
\end{pmatrix}. 
\]

\begin{prop}[\cite{AT}, Proposition 2.3]
\label{propMv}
The Mukai vector induces an isometry between the orthogonal complement $A_2^{\perp}$ of $A_2$ in $\tilde{H}(\A_Y,\Z)$ and the primitive lattice $\langle h \rangle^{\perp}=H^4(Y,\Z)_{\text{prim}}(1)$. Moreover, if $\mathsf{k}_1, \dots, \mathsf{k}_n$ are elements of $\tilde{H}(\A_Y,\Z)$, then $v_M$ induces an isometry 
$$\langle \lambda_1, \lambda_2, \mathsf{k}_1, \dots, \mathsf{k}_n \rangle^{\perp} \cong \langle h, c_2(\mathsf{k}_1), \dots, c_2(\mathsf{k}_n) \rangle^{\perp}.$$
\end{prop} 

\begin{rmk}
\label{rmksc}
Since, by definition, the lattice $A_2$ is contained in $N(\A_Y)$, the orthogonality condition implies that $T(\A_Y)$ is in $A_2^{\perp}$. In particular, as observed in \cite{Huy}, Section 3.3, the orthogonal complement to the trascendental lattice in $A_2^{\perp}$ is $N(A_Y) \cap A_2^{\perp}$.
\end{rmk}

\begin{thm}[\cite{AT}, Theorem 1.1]
\label{AddThom}
If $\A_Y$ is geometric, then $Y$ belongs to $\mathcal{C}_d$ for some $d$ satisfying condition \emph{\textbf{(a)}} of Theorem \ref{propHass}. Conversely, for each $d$ satisfying \emph{\textbf{(a)}}, the set of cubic fourfolds $Y$ in $\mathcal{C}_d$ for which $\A_Y$ is geometric forms a Zariski open dense subset.
\end{thm}

\begin{rmk}
\label{rmk_BLM+}
In \cite{BLM+}, Corollary 1.7, the authors prove that every $Y \in \mathcal{C}_d$ for $d$ satisfying {\textbf{(a)}} is geometric, extending Addington and Thomas' result to the whole divisor. 
\end{rmk}

In \cite{Huy}, Proposition 3.4, Huybrechts proved that, given two cubic fourfolds $Y$ and $Y'$, the existence of a Fourier-Mukai equivalence $\A_Y \xrightarrow{\sim} \A_{Y'}$ implies the existence of a Hodge isometry of the corresponding Mukai lattices. The surprising fact is that, under some assumptions, the category $\A_Y$ is completely determined by the Hodge structure on $\tilde{H}(\A_Y,\Z)$, as we recall in the next section.   


\subsection{Associated twisted K3 surface}
In \cite{Huy}, Huybrechts generalized Theorem \ref{propHass} and Theorem \ref{AddThom} to the case of cubic fourfolds admitting an \emph{associated twisted K3 surface}. We recall that a twisted K3 surface is the data of a K3 surface $X$ and a class in the Brauer group $H^2(X,\mathcal{O}^{\ast}_X)_{\text{tors}}$ of $X$. Following \cite{HuySt0}, Section 2, let $B$ be a rational class of $H^2(X,\Q)$, which is sent to $\alpha$ through the composition
$$H^2(X,\Q) \rightarrow H^2(X,\mathcal{O}_X) \xrightarrow{\text{exp}} H^2(X,\mathcal{O}^{\ast}_X).$$ 
We say that $B$ is a B-field lift of $\alpha$. We denote by $\tilde{H}(X,\alpha,\Z)$ the cohomology ring $H^*(X,\Z)$ with the Mukai pairing and the weight two Hodge structure defined by
$$\tilde{H}^{2,0}(X,\alpha):=\text{exp}(B)H^{2,0}(X) \quad \text{and} \quad \tilde{H}^{1,1}(X,\alpha):=\text{exp}(B)H^{1,1}(X).$$ 
We see that $\tilde{H}(X,\alpha,\Z)$ is isomorphic as a lattice to $\tilde{\Lambda}$ and we call it the Mukai lattice of $(X,\alpha)$. We can consider the algebraic part
$$N(X,\alpha)=\tilde{H}^{1,1}(X,\alpha,\Z):=\tilde{H}^{1,1}(X,\alpha) \cap \tilde{H}(X,\alpha,\Z)$$
and we define the generalized twisted trascendental lattice $T(X,\alpha)$ as the orthogonal complement of $N(X,\alpha)$ with respect to the Mukai pairing. On the other hand, using the intersection product with $B$, we can identify the class $\alpha$ with a surjective morphism $\alpha: T_X \rightarrow \Z/ \text{ord}(\alpha)\mathbb{Z}$. Then, the kernel of $\alpha$ is isomorphic via $\text{exp}(B)$ to $T(X,\alpha)$ (see \cite{Huyv}, Proposition 4.7). For this reason, we will use the same notation for $T(X,\alpha)$ and $\ker(\alpha)$ (resp.\ for $N(X,\alpha)$ and the orthogonal complement of $\ker(\alpha)$ in $\tilde{H}(X,\alpha,\Z)$), even if the first one is primitively embedded in $\tilde{H}(X,\alpha,\Z)$, while the second one is not.

As in the untwisted case, the condition of having an associated twisted K3 surface on the level of Hodge structures on the Mukai lattices depends only on the value of the discriminant $d$.  

\begin{thm}[\cite{Huy}, Theorem 1.3]
\label{exK3twist}
Let $Y$ be a cubic fourfold. There exist a twisted K3 surface $(X, \alpha)$ and a Hodge isometry $\tilde{H}(\A_Y, \Z) \cong \tilde{H}(X, \alpha, \Z)$ if and only if $Y$ belongs to $\mathcal{C}_d$ for $d$ such that
\begin{equation}\tag{\textbf{a'}}
\label{condd}
n_i \equiv 0 (\emph{mod}\,2) \text{ for all primes } p_i \equiv 2 (\emph{mod}\,3) \text{ in } 2d=\prod p_i^{n_i}.
\end{equation}
\end{thm} 

Moreover, Theorem \ref{AddThom} and Remark \ref{rmk_BLM+} have the following form in the twisted setting.
\begin{itemize}
\item If there exists a twisted K3 surface $(X,\alpha)$ such that the category $\A_Y$ is equivalent to the derived category $\D(X,\alpha)$ of bounded complexes of $\alpha$-twisted coherent sheaves on $X$, then the cubic fourfold $Y$ belongs to $\mathcal{C}_d$ for $d$ satisfying condition \eqref{condd} of Theorem \ref{exK3twist} (see \cite{Huy}, Theorem 1.4(i)).
\item In \cite{Huy}, Theorem 1.4(ii), Huybrechts proved that if $d$ satisfies \eqref{condd}, then a Zariski open subset of cubic fourfolds $Y$ in the divisor $\mathcal{C}_d$ have $\A_Y \xrightarrow{\sim} \D(X,\alpha)$. In \cite{BLM+}, Proposition 32.2, the authors extend this result to all cubic fourfolds in $\mathcal{C}_d$.
\end{itemize}

Finally, we recall the following result, crucial for the proofs in the next sections, which states that the Hodge structure on the Mukai lattice determines the K3 category for cubic fourfolds with (twisted) associated K3 surface. 

\begin{thm}[\cite{Huy}, Theorem 1.5(ii)]
\label{thm_Huy}
For $d$ satisfying \eqref{condd} and a Zariski dense open set of cubic fourfolds $Y \in \mathcal{C}_d$, there exists a Fourier-Mukai equivalence $\A_Y \xrightarrow{\sim} \A_{Y'}$ if and only if there exists a Hodge isometry $\tilde{H}(\A_Y,\Z) \cong \tilde{H}(\A_{Y'},\Z)$. 
\end{thm}

\begin{rmk}
\label{rmk_BLMS+Huy}
Using that every $Y$ in such a divisor has $\A_Y \xrightarrow{\sim} \D(X,\alpha)$ by \cite{BLM+}, we can extend this result to all the divisor $\mathcal{C}_d$ following the same proof of \cite{Huy}.
\end{rmk}

\begin{rmk}
\label{rmk_pp}
Set
$$\tilde{Q}:=\lbrace \varphi \in \p(\tilde{\Lambda} \otimes \C): (\varphi,\varphi)=0, (\varphi,\bar{\varphi})>0 \rbrace.$$
A point $\varphi \in \tilde{Q}$ is of K3 type (resp.\ of twisted K3 type) if there is a K3 surface $X$ (resp.\ a twisted K3 surface $(X,\alpha)$) such that the Hodge structure defined by $\varphi$ on $\tilde{\Lambda}$ is Hodge isometric to $\tilde{H}(X,\Z)$ (resp.\ $\tilde{H}(X,\alpha,\Z)$) (see \cite{Huy}, Definition 2.5). We denote by $Q_{\text{K3}}$ (resp.\ $Q_{\text{K3}'}$) the set of points of K3 type (resp.\ of twisted K3 type) in $\tilde{Q}$. 

Notice that $\mathcal{D}' \subset Q \subset \tilde{Q}$, as $L^0 \cong A_2^{\perp}$. Thus, we can consider the sets
$$\mathcal{D}_{\text{K3}}:= Q_{\text{K3}} \cap \mathcal{D'} \quad \text{and} \quad \mathcal{D}_{\text{K3}'}:= Q_{\text{K3}'} \cap \mathcal{D'},$$
containing period points in $\mathcal{D}'$ of (twisted) K3 type (see \cite{Huy}, Section 2.5).
\end{rmk}

\subsection{Counting formulas for Fourier-Mukai partners of a K3 surface}
The aim of this subsection is to recollect some known formulas which count the number of isomorphism classes of (twisted) Fourier-Mukai partners of a given (twisted) K3 surface. We recall that a \emph{twisted Fourier-Mukai partner} of a K3 surface $X$ (resp.\ a twisted K3 surface $(X,\alpha)$) is a twisted K3 surface $(X',\alpha')$ such that there exists an equivalence of categories $\D(X) \xrightarrow{\sim} \D(X',\alpha')$ (resp.\ $\D(X,\alpha) \xrightarrow{\sim} \D(X',\alpha')$); if the Brauer class $\alpha'$ is trivial, we say that the Fourier-Mukai partner is untwisted. 

The first result concerns the number of isomorphism classes of untwisted Fourier-Mukai partners of a very general polarized K3 surface, which is determined by the number of distinct primes in the factorization of the degree of the polarization class.  

\begin{thm}[\cite{Og}, Proposition 1.10]
\label{exFMk3}
Let $X$ be a K3 surface with Néron-Severi lattice $\text{NS}(X)$ of rank one generated by a polarization class $l_X$ such that $l_X^2=2n$. Let $m$ be the number of (isomorphism classes of) Fourier-Mukai partners of $X$; then we have:
\begin{itemize}
\item $m=1$, if $l_X^2=2$ or $l_X^2=2^a$,
\item $m=2^{h-1}$, if $l_X^2=2 p_1^{e_1}\cdot \dots p_h^{e_h}$,
\item $m=2^{h}$, if $l_X^2=2^a p_1^{e_1}\cdot \dots p_h^{e_h}$, 
\end{itemize}
where $a, h$ and the $e_i$'s are natural numbers with $a \geq 2$, the $p_i$'s are different primes such that $p_i \geq 3$.
\end{thm}

More generally, Ma proved in \cite{Ma} a counting formula for isomorphism classes of twisted Fourier-Mukai partners of a twisted K3 surface $(X,\alpha)$ which admits an untwisted Fourier-Mukai partner (see \cite{Ma}, Theorem 1.1). Moreover, relaxing this hypothesis, he obtained an upper bound to the number of twisted Fourier-Mukai partners of $(X,\alpha)$. We conclude this paragraph by resuming Ma's construction, which will be useful in Section 4.

Let $(X,\alpha)$ be a twisted K3 surface with $\text{ord}(\alpha)= \kappa$. We recall that a twisted K3 surface $(X',\alpha')$ is isomorphic to $(X,\alpha)$ if there exists an isomorphism $F: X \cong X'$ such that $F^*\alpha'=\alpha$. We denote by $\FM^r(X,\alpha)$ the set of isomorphism classes of Fourier-Mukai partners $(X',\alpha')$ of $(X,\alpha)$, having $\alpha'$ of order $r$. We say that $(X_1,\alpha_1)$ and $(X_2,\alpha_2)$ in $\FM^r(X,\alpha)$ are $\sim$-equivalent if there exists a Hodge isometry $g: T_{X_1} \cong T_{X_2}$ such that $g^*\alpha_2=\alpha_1$. We define the quotient
$$\mathcal{FM}^r(X,\alpha):= \FM^r(X,\alpha)/\sim$$
and we denote by $\pi: \FM^r(X,\alpha) \twoheadrightarrow \mathcal{FM}^r(X,\alpha)$ the quotient map. Let $I^r(d(T(X,\alpha)))$ be the set of all isotropic elements of order $r$ of the discriminant group $(d(T(X,\alpha),q_{T(X,\alpha)})$ of $T(X,\alpha)$, i.e.
$$I^r(d(T(X,\alpha))):= \left\lbrace x \in d(T(X,\alpha)): q_{T(X,\alpha)}(x)=0 \in \Q/ 2\mathbb{Z}, \text{ord}(x)=r \right\rbrace.$$
We define the map
\begin{equation}
\label{mu}
\mu: \mathcal{FM}^r(X,\alpha) \rightarrow \text{O}_{\text{Hdg}}(T(X,\alpha)) \setminus I^r(d(T(X,\alpha))),
\end{equation}
where $\text{O}_{\text{Hdg}}(T(X,\alpha))$ is the group of Hodge isometries of the generalized trascendental lattice, in the following way. For every $(X_1,\alpha_1)$ in $\FM^r(X,\alpha)$, there exists a Hodge isometry $g_1: T(X_1,\alpha_1) \cong T(X,\alpha)$. Then 
$$\frac{g_1^{\vee}(T_{X_1})}{T(X,\alpha)} \cong \frac{T_{X_1}}{T(X_1,\alpha_1)} \cong \frac{\Z}{r\mathbb{Z}}$$
is an isotropic, cyclic subgroup of $d(T(X,\alpha))$ of order $r$. Thus, for every class $[(X_1,\alpha_1)]$ in $\mathcal{FM}^r(X,\alpha)$, we set
$$\mu([(X_1,\alpha_1)])=x:=[g_1(\alpha^{-1}_1(\bar{1}))] \in \text{O}_{\text{Hdg}}(T(X,\alpha)) \setminus I^r(d(T(X,\alpha))).$$
We have that:
\begin{enumerate}
\item The map $\mu$ is well-defined and injective (see \cite{Ma}, Lemma 3.2);
\item The image of $\mu$ is contained in $\text{O}_{\text{Hdg}}(T(X,\alpha)) \setminus J^r(d(T(X,\alpha)))$, where
$$J^r(d(T(X,\alpha)))=\left\lbrace x \in I^r(d(T(X,\alpha))): \text{there exists an embedding }U \hookrightarrow \langle N(X,\alpha), \lambda(x) \rangle \right\rbrace,$$ 
for $\lambda: d(T(X,\alpha)) \cong d(N(X,\alpha))$ (see \cite{Ma}, Proposition 3.4).
\end{enumerate}

On the other hand, for every $(X_1, \alpha_1)$ in $\FM^r(X,\alpha)$, we can define a map
\begin{equation}
\label{nu}
\nu: \pi^{-1}(\pi(X_1,\alpha_1)) \rightarrow \Gamma(X_1,\alpha_1)^+ \setminus \text{Emb}(U,N(X_1)),
\end{equation}
where $\text{Emb}(U,N(X_1))$ is the set of the embeddings of $U$ in $N(X_1)=H^0(X_1,\Z) \oplus \text{NS}(X_1) \oplus H^4(X_1,\Z)$ and $\Gamma(X_1,\alpha_1)^+$ is the set of orientation-preserving isometries of $N(X_1)$ \footnote{In general, given a lattice $L$ of signature $(l_+,l_-)$ with $l_+>0$, we can consider the set of oriented positive definite $l_+$-planes in $L \otimes \R$. An orientation for $L$ is the choice of an orientation for such a positive definite $l_+$-plane. For a subgroup $\Gamma$ of $\text{O}(L)$, we denote by $\Gamma^+$ the subgroup of isometries of $\Gamma$ which preserve the given orientation (see \cite{Ma}, Section 2.1). 
}, which come from isometries of $T_{X_1}$ fixing $\alpha_1$ (see \cite{Ma}, Section 3.2). We have that:
\begin{enumerate}
\item The map $\nu$ is injective (see \cite{Ma}, Lemma 3.5);
\item The map $\nu$ is surjective if and only if the C\v{a}ld\v{a}raru's Conjecture holds (see \cite{Ma}, Remark 3.7).
\end{enumerate}
We recall the statement of \emph{C\v{a}ld\v{a}raru's Conjecture}, which was proposed for the first time in \cite{Cald}, Conjecture 5.5.5.

\begin{conj}[\cite{Ma}, Question 3.8]
\label{Caldconj}
Let $(X,\alpha)$ be a twisted K3 surface. For each untwisted Fourier-Mukai partner $X'$ of $X$ and each Hodge isometry $g: T_{X'}\cong T_{X}$, the twisted K3 surface $(X', g^*\alpha)$ is a Fourier-Mukai partner of $(X,\alpha)$.
\end{conj}

\begin{rmk}
We point out that Conjecture \ref{Caldconj} is related to an other conjecture due to C\v{a}ld\v{a}raru, which asks whether two twisted K3 surfaces having Hodge isometric twisted trascendental lattices are Fourier-Mukai partners. This conjecture is known to be false in general by \cite{HuySt0}, Example 4.11.
\end{rmk}

To state Ma's formula, we need to introduce some notation. For every $x$ in $I^r(d(T(X,\alpha)))$, we define the overlattice 
$$T_x:=\langle x,T(X,\alpha) \rangle$$
of $T(X,\alpha)$ and the morphism
$$\alpha_x : T_x \twoheadrightarrow \frac{T_x}{T(X,\alpha)} \cong \langle x \rangle \cong \frac{\Z}{r\mathbb{Z}}.$$
For a pair $(x,M)$ such that 
$$\langle \lambda(x), N(X,\alpha) \rangle \cong U \oplus M,$$ 
we define the number 
$$\tau(x,M):= \# (\text{O}_{\text{Hdg}}(T_x,\alpha_x) \setminus \text{O}(d(M)) / \text{O}(M)),$$
where $\text{O}_{\text{Hdg}}(T_x,\alpha_x)$ is the set of Hodge isometries $g$ of $T_x$, such that $g^*\alpha_x=\alpha_x$.
For a natural number $r$, we define
$$ \varepsilon(r)=\begin{cases} 1, & \mbox{if } r=1,2 \\ 2, & \mbox{if } r \geq 3.
\end{cases}$$
Finally, if $\mathcal{G}(L)$ is the genus of a lattice $L$, $\text{O}(L)_0$ is the kernel of the map $r_L: \text{O}(L) \rightarrow \text{O}(d(L))$ and $\text{O}(L)_0^+$ is the subgroup of $\text{O}(L)_0$ of orientation-preserving isometries, we define the subsets
$$\mathcal{G}_1(L):=\lbrace L' \in \mathcal{G}(L): \text{O}(L')_0^+\neq\text{O}(L')_0 \rbrace, \quad \mathcal{G}_2(L):=\lbrace L' \in \mathcal{G}(L): \text{O}(L')_0^+=\text{O}(L')_0 \rbrace.$$
Using the previous observations, Ma proved that the following inequality holds.

\begin{thm}[\cite{Ma}, Proposition 4.3]
\label{Maformula}
We have the inequality
\begin{equation}
\label{CF}
\#\emph{FM}^r(X,\alpha) \leq \sum_x \left\lbrace \sum_M \tau(x,M) + \varepsilon(r) \sum_{M'} \tau(x,M') \right\rbrace.
\end{equation}  
Here:
\begin{itemize}
\item $x$ runs over the set $\emph{O}_{\text{Hdg}}(T(X,\alpha)) \setminus J^r(d(T(X,\alpha)))$;
\item the lattices $M$ and $M'$ run over the sets $\mathcal{G}_1(M_{\varphi})$, $\mathcal{G}_2(M_{\varphi})$ respectively, where $M_{\varphi}$ is a lattice satisfying $\langle \lambda(x), N(X,\alpha) \rangle \cong U \oplus M_{\varphi}$. 
\end{itemize} 
\end{thm}
 
\section{Construction of the examples (untwisted case)}

The aim of this section is to prove Theorem \ref{thmmio}. In the first paragraph we make some preliminary remarks, while in the last section we provide the proof of the theorem.

\subsection{Some preliminary computations}  

Let $Y$ be a cubic fourfold in $\mathcal{C}_d$ with $d$ satisfying condition \textbf{(a)} of Theorem \ref{propHass} and $\text{rk}H^{2,2}(Y,\Z)=2$; let us choose a K3 surface $X$ of degree $d$ associated to $Y$. In this section we show that the $m$ non-isomorphic representatives of the isomorphism classes of untwisted Fourier-Mukai partners of $X$ determine $m$ (resp.\ $\lceil m/2 \rceil$) distinct points in the period domain $\mathcal{D}_d$ if $d \equiv 2 (\text{mod}\,6)$ (resp.\ $d \equiv 0 (\text{mod}\,6)$).

We recall that $N(\A_Y)$ has rank $3$, because $Y$ has $\text{rk}H^{2,2}(Y,\Z)=2$ (see \cite{Huy}, Lemma 2.2). Let $\mathsf{v}_Y$ be a generator of the rank one lattice $N(\A_Y) \cap A_2^{\perp}$. Let $m$ be the number of isomorphism classes of Fourier-Mukai partners of $X$. We fix a representative for each class of isomorphism and we denote them by $X_1,\dots,X_m$, choosing $X_1:=X$. By \cite{Or}, Theorem 3.3, this is equivalent to ask that, for every index $2 \leq k \leq m$, there exists a Hodge isometry $\tilde{H}(X,\Z)\cong \tilde{H}(X_k,\Z)$. In particular, the Néron-Severi lattice of $X_k$ has rank one with the polarization class of degree $d$. We denote by $x_k$ the point in the period domain $\mathcal{N}_d^{\prime}$, which is determined by the Hodge structure on $H^2(X_k,\Z)$. These points also descend to different points in the period domain $\mathcal{N}_d$, since they come from non-isomorphic polarized K3 surfaces. 

Composing the isometries of Proposition \ref{propMv} and of Theorem \ref{propHass}, 
we get the isometry of Hodge structures
$$\varphi: T(\A_Y)= \langle \lambda_1, \lambda_2, \mathsf{v}_Y \rangle^{\perp} \cong H^2(X,\Z)_{\text{prim}}=T_X.$$
This induces an isomorphism 
$$j^{\prime}: \mathcal{D}_d^{\prime} \rightarrow \mathcal{N}_d^{\prime}$$
between the local period domains. For every $1 \leq k \leq m$, we denote by $y_k$ the preimage of $x_k$ with respect to $j^{\prime}$. By definition, the point $y_k$ parametrizes a special Hodge structure with labelling of discriminant $d$ on $A_2^{\perp}$. In particular, there exists a class $\mathsf{v}_k$ in $A_2^{\perp}$ with $(\mathsf{v}_k,\mathsf{v}_k)=(\mathsf{v}_Y,\mathsf{v}_Y)$, such that if $T_k=(\Z\mathsf{v}_k)^{\perp}$ in $A_2^{\perp}$, then there is an isometry of Hodge structures $\varphi_k: T_{X_k} \cong T_k$.

By Theorem \ref{Immmodspace}, the isomorphism $j^{\prime}$ descends to an isomorphism
$$j: \mathcal{D}_d^{\text{mar}} \rightarrow \mathcal{N}_d.$$
Thus, the points $y_1,\dots,y_m$ descends to distinct points, which we denote in the same way, in the period domain $\mathcal{D}_d^{\text{mar}}$.

Let us consider their images in the period domain $\mathcal{D}_d^{\text{lab}}$; here, these points could be identified. We observe that, if some of them are not identified in $\mathcal{D}_d^{\text{lab}}$, then they correspond to distinct points in the period domain $\mathcal{D}_d$. Indeed, the map sending $\mathcal{D}_d^{\text{lab}}$ in $\mathcal{D}_d$, which forgets the labelling, is generically an isomorphism.

In particular, it is enough to study the behavior of the forgetful map $\rho: \mathcal{D}_d^{\text{mar}} \rightarrow \mathcal{D}_d^{\text{lab}}$ over the points $y_1,\dots,y_m$, to understand how many of them define different special Hodge structures of discriminant $d$. According to Proposition \ref{markedvslab}, we have to distinguish two cases depending on the value of the discriminant.\\

\noindent \textbf{Case $d \equiv 2 (\text{mod}\,6)$:} by Theorem \ref{Immmodspace}, we have that the map $\rho$ is an isomorphism. Hence, $y_1,\dots,y_m$ are not identified by the action of $\Gamma_d^{+}$ and they determine $m$ distinct special Hodge structures of discriminant $d$. \\
  
\noindent \textbf{Case $d \equiv 0 (\text{mod}\,6)$:} by Theorem \ref{Immmodspace}, the map $\rho$ is a double cover. Hence, our points $y_1,\dots,y_m$ descend to at least $\lceil m/2 \rceil$ different Hodge structures in $\mathcal{D}_d$.

On the other hand, we observe that if $T$ is a sublattice of $\Lambda$ which is Hodge isometric to $T_X$, then the lattice $\gamma(T)$, with the Hodge structure induced by the one on $T$ through $\gamma_{\C}$, satisfies the same property. As a consequence, we obtain that the K3 surface $X_{\gamma(T)}$ with trascendental lattice $\gamma(T)$ is a Fourier-Mukai partner of $X$.  Let $X_T$ be the K3 surface with trascendental lattice $T$ and consider $X_T$ and $X_{\gamma(T)}$ with the unique polarization. Note that $X_T$ and $X_{\gamma(T)}$ are not isomorphic as polarized K3 surfaces, because $\gamma$ is not in $G_d^+ \cong  \Sigma_d^+$ by Remark \ref{rmkgamma} and Theorem \ref{Immmodspace}. Thus, their corresponding period points in $\mathcal{N}_d$ define two distinct period points in $\mathcal{D}_d^{\text{mar}}$, which belong to the same fiber of $\rho$. 

As a consequence, the $m$ points $y_1,\dots,y_m$ determine exactly $\lceil m/2 \rceil$ different special Hodge structures of discriminant $d$.  


\subsection{Proof of Theorem \ref{thmmio}}

Keeping the notation introduced in Section 3.1, we set
\[p:=
\begin{cases}
m & \text{if } d \equiv 2 (\text{mod}\,6) \\
\lceil m/2 \rceil & \text{if } d \equiv 0 (\text{mod}\,6).
\end{cases} 
\]
Firstly, we prove that $p$ is an upper bound to the number of Fourier-Mukai partners of $Y$. Actually, this represents an alternative way to prove the finiteness result of \cite{Huy}, Corollary 3.5, for a cubic fourfold with $H^{2,2}$ of rank $2$ and associated K3 surface.

\begin{prop}
\label{propmia2}
Let $Y$ be a cubic fourfold in $\mathcal{C}_d$ with $d$ satisfying condition \emph{\textbf{(a)}} of Theorem \ref{propHass} and with $\emph{rk}H^{2,2}(Y,\Z)=2$. If the associated K3 surface $X$ admits $m$ non isomorphic Fourier-Mukai partners, then the cubic fourfold $Y$ cannot have more than $m$ (resp.\ $\lceil m/2 \rceil$) Fourier-Mukai partners if $d \equiv 2 (\emph{mod}\,6)$ (resp.\ if $d \equiv 0 (\emph{mod}\,6)$). 
\end{prop}
\begin{proof}
Consider the $p$ distinct points $y_1,\dots ,y_p \in \mathcal{D}_d$ defined in Section 3.1. We claim that $y_k$ belongs to the image of the period map of cubic fourfolds for every $1 \leq k \leq p$. Indeed, we observe that $d$ is not $2$ or $6$, because $d$ satisfies condition \eqref{0}, as $\mathcal{C}_d$ is not empty. Moreover, the point $y_k$ has a unique labelling, as recalled in Section 2.2. It follows that $y_k$ is a period point in the complement of $\mathcal{D}_2 \cup \mathcal{D}_6$. By \cite{Laza}, Theorem 1.1, there exists a cubic fourfold $Y_k$ in $\mathcal{C}_d$ such that $\tau(Y_k)=y_k$, as we wanted.  

Now, let $Y'$ be a Fourier-Mukai partner of $Y$, i.e.\ such that there exists an equivalence $\A_Y \xrightarrow{\sim} \A_{Y'}$ of Fourier-Mukai type. By \cite{Huy}, Proposition 3.4, this induces a Hodge isometry $\tilde{H}(\A_Y,\Z) \cong \tilde{H}(\A_{Y^{'}},\Z)$. Notice that $\text{rk}H^{2,2}(Y',\Z)=2$, as for $Y$. Thus its period point $\tau(Y') \in \mathcal{D}_d$ corresponds to a point in $\mathcal{D}_d^{\text{lab}}$ which we denote in the same way. Let $y' \in \mathcal{D}_d^{\text{mar}}$ be a point in the fiber $\rho^{-1}(\tau(Y'))$. We set $x':=j(y') \in \mathcal{N}_d$. The point $x'$ corresponds to a very general K3 surface $X'$ with unique primitive polarization $l_{X'}$ of degree $d$. In particular, since $\tilde{H}(\A_{Y^{'}},\Z)$ is Hodge isometric to the Mukai lattice $\tilde{H}(X',\Z)$, it follows from \cite{Or}, Theorem 3.3 that $X'$ is a Fourier-Mukai partner of $X$. Thus, there exists an index $k \in \lbrace 1, \dots, m \rbrace$ such that, if $l_{X_k}$ denotes the unique primitive polarization on $X_k$, then $(X',l_{X'}) \cong (X_k,l_{X_k})$ as polarized K3 surfaces. Equivalently, the points $x'$ and $x_k$ are identified in $\mathcal{N}_d$. Since $j$ is an isomorphism, it follows that $y_k=y'$ in $\mathcal{D}_d^{\text{mar}}$. In particular, they represent the same point in $\mathcal{D}_d$: by the Torelli Theorem for cubic fourfolds, we conclude that $Y'$ is isomorphic to $Y_k$. This implies the desired statement.
\end{proof} 

We are ready to prove Theorem \ref{thmmio}, which is formulated in a more precise way using Theorem \ref{exFMk3}. 

\begin{prop}[Theorem \ref{thmmio}]
\label{propmia}
Let $d$ be a positive integer satisfying conditions \emph{(0)} and \emph{\textbf{(a)}}. Then, the number of isomorphism classes of Fourier-Mukai partners for a very general special cubic fourfold in $\mathcal{C}_d$ is
\begin{itemize}
\item $p=2^{h-1}$, if $d \equiv 2 (\emph{mod}\,6)$ and the prime factorization of $d$ has $h>1$ distinct odd primes;
\item $p=2^{h-2}$, if $d \equiv 0 (\emph{mod}\,6)$ and the prime factorization of $d$ has  $h>2$ distinct odd primes;
\item $p=1$, otherwise.
\end{itemize}
\end{prop}
\begin{proof}
Let $Y$ be a very general special cubic fourfold in $\mathcal{C}_d$ as in the statement. We consider the $p$ distinct points $y_1,\dots,y_p$ in $\mathcal{D}_d$ defined in Section 3.1. We claim that there exist $p$ very general special cubic fourfolds $Y_1, \dots, Y_p \in \mathcal{C}_d$ such that $\tau(Y_k)=y_k$ for $k=1,\dots,p$. Indeed, $d$ is not $2$ or $6$, because $d$ satisfies condition \eqref{0}. Moreover, every $y_k$ has a unique labelling, since they are very general. It follows that $y_k$ is a period point in the complement of $\mathcal{D}_2 \cup \mathcal{D}_6$. By \cite{Laza}, Theorem 1.1, we deduce our claim. Note that $Y_1 \cong Y$ and the cubic fourfolds $Y_1,\dots,Y_p$ are not isomorphic to each other by Torelli Theorem for cubic fourfolds. 

By construction, for every $2 \leq k \leq p$, there is an isometry of Hodge structures
\begin{equation*}
\tilde{H}(\A_Y,\Z) \cong \tilde{H}(X,\Z) \cong \tilde{H}(X_k,\Z) \cong \tilde{H}(\A_{Y_k},\Z).
\end{equation*}
By Theorem \ref{thm_Huy}, the existence of such an isometry of Hodge structures implies the existence of a Fourier-Mukai equivalence between $\A_Y$ and $\A_{Y_k}$. On the other hand, by Proposition \ref{propmia2} every other Fourier-Mukai partner of $Y$ is isomorphic to one of those we constructed. Finally, the counting formula of Theorem \ref{exFMk3} implies the statement.
\end{proof}

\begin{rmk}
\label{rmk_extwithBLM+}
As explained in Remark \ref{rmk_BLMS+Huy},  using \cite{BLM+} we can remove the assumption that $Y$ is very general special to Theorem \ref{thmmio} and prove the result for every cubic fourfold with $H^{2,2}(Y,\Z)$ of rank two and associated K3 surface.
\end{rmk}

\begin{ex}
Using Proposition \ref{propmia}, it is easy to find the divisors in $\mathcal{C}$ whose very general element has non trivial Fourier-Mukai partners. For example, take $d=182$, which is $\equiv 2 (\text{mod}\,6)$. By Proposition \ref{propmia} the very general cubic fourfold in $\mathcal{C}_{182}$ has one non isomorphic Fourier-Mukai partner. If $d= 546 \equiv 0 (\text{mod}\,6)$, then the very general element in $\mathcal{C}_{546}$ has one non isomorphic Fourier-Mukai partner.  
\end{ex}

\begin{rmk}
Notice that, to prove these results, we have fixed an associated K3 surface to $Y$ and, consequently, an isomorphism between the period domains $\mathcal{D}_d^{\text{mar}}$ and $\mathcal{N}_d$. Actually, we could choose a Fourier-Mukai partner of $X$ as fixed associated K3 surface to $Y$: this would have given a different isomorphism $\tilde{j}$ on the level of period domains and a different identification of Fourier-Mukai partners of $Y$ with Fourier-Mukai partners of $X$ (see \cite{Hass2}, Remark 27). 
\end{rmk}

\section{Construction of the examples (twisted case)}

This section is devoted to the proof of Theorem \ref{propmiatwist}. In particular, in Section 4.2 and 4.3 we explicit the lower bound to the number of Fourier-Mukai partners of a cubic fourfold $Y$ as in Theorem \ref{propmiatwist}, in terms of the number of primes in the prime factorization of the discriminant of $Y$ and the Euler function evaluated in the order of the Brauer class of the associated twisted K3 surface.  

\subsection{Proof of Theorem \ref{propmiatwist}}

Let $Y$ be a very general special cubic fourfold in $\mathcal{C}_d$ such that condition \textbf{(a')} of Theorem \ref{exK3twist} holds. If $d$ satisfies in addition \textbf{(a)}, then we fix an associated untwisted K3 surface and the following construction provides the same period points constructed in Section 3. In the general case, the cubic fourfold $Y$ has a twisted associated K3 surface, which we denote by $(X,\alpha)$ with $\alpha$ of order $\kappa$. Assume also that
\begin{equation}\tag{\textbf{b}}
9 \text{ does not divide the discriminant } d.
\end{equation}
Condition \textbf{(b)} implies that the discriminant group of $T(\A_Y)$ and, consequently, also that of $T(X,\alpha)$, are cyclic, by Proposition \ref{propMv} and Proposition \ref{K_dperp}. As a consequence, by \cite{Ni}, Theorem 1.14.4, the natural embedding
\begin{equation}
\label{eq_T}
T(X,\alpha) \hookrightarrow \tilde{H}(X,\alpha,\Z) \cong \tilde{\Lambda}
\end{equation}
is unique up to isometry of $\tilde{\Lambda}$, because $\text{rk}(N(X,\alpha)) \geq l(d(T(X,\alpha)))+2=3$ \footnote{Assumption \textbf{(b)} is used in the proof of Lemma \ref{lemmafondam}, to lift an isometry of the twisted trascendental lattices acting trivially on the discriminant group to the trascendental lattices of two twisted K3 surfaces. Also  it is used in the proof of Proposition \ref{thmcasogen}, to lift an isometry of the twisted trascendental lattices to the Mukai lattices of two twisted K3 surfaces.}.
   
Now, let $(X',\alpha')$ be a twisted Fourier-Mukai partner of $(X,\alpha)$ of the same order $\kappa$. By \cite{HuySt0}, Proposition 4.3, there is an isometry of Hodge structures $\tilde{H}(X,\alpha,\Z) \cong \tilde{H}(X',\alpha',\Z)$ preserving the orientation. Thus the choice of $(X',\alpha')$ determines a Hodge structure on the abstract lattice $\tilde{\Lambda}$ which is Hodge isometric to $\tilde{H}(\A_X,\Z)$. Roughly speaking, in order to produce a Fourier-Mukai partner of $Y$ out of $(X',\alpha')$, we need to check whether the two Hodge structures on $\tilde{H}(\A_X,\Z)$, induced by $(X,\alpha)$ and $(X',\alpha')$ respectively, are not Hodge isometric on $A_2^{\perp}$.

To explain better the previous sentence, we need to introduce the following notation. Let $y'$ be the period point in the quadric $Q$ defined in \eqref{quadric} parametrizing the Hodge structure on $\tilde{H}(\A_X,\Z)$ induced by $(X',\alpha')$. Up to exchanging $\tilde{H}^{2,0}(X',\alpha')$ with $\tilde{H}^{0,2}(X',\alpha')$, we can assume that $y'$ is in $\mathcal{D}_d'$.  Recall that the image of $y'$ in $\mathcal{D}_d$ (which we still denote by $y'$) is equal to the period point $y:=\tau(Y)$ if and only if there is an isometry of $K_d^{\perp}$ which extends to an isometry of $L^0$. In this case, we would have that the two Hodge structures on $K_d^{\perp}$ given by those on $\tilde{H}(X,\alpha,\Z)$ and $\tilde{H}(X',\alpha',\Z)$, respectively, induce the same Hodge structure on $L^0$, or equivalently on $A_2^{\perp}$. 

As in the untwisted case, it is convenient to consider firstly the period domain $\mathcal{D}_d^{\text{mar}}$. Here, the points $y$ and $y'$ are identified if and only if they are in the same orbit by the action of $G_d^+$. We recall that elements in $G_d^+$ are isometries of $K_d^{\perp}$ acting trivially on the discriminant group $d(K_d^{\perp})$.
   
Assume that $(X',\alpha')$ is not isomorphic to $(X,\alpha)$. In the next lemma, we prove that $y$ and $y'$ are distinct in the period domain $\mathcal{D}_d^{\text{mar}}$ under this assumption. 


\begin{lemma}
\label{lemmafondam}
The period points $y$ and $y'$ are distinct in $\mathcal{D}_d^{\emph{mar}}$.
\end{lemma} 
\begin{proof}
We will actually prove that if $y=y'$ in $\mathcal{D}_d^{\text{mar}}$, then the twisted K3 surfaces $(X,\alpha)$ and $(X',\alpha')$ are isomorphic, in contradiction with our assumption.

If $y$ and $y'$ are the same point in the period domain $\mathcal{D}_d^{\text{mar}}$, then there exists an isometry of Hodge structures
$$\eta: T(X,\alpha) \cong T(X',\alpha'),$$
such that the induced isomorphism $\bar{\eta}$ between the discriminant groups $d(T(X,\alpha))$ and $d(T(X',\alpha'))$ is trivial. 

First of all, we prove that the Hodge isometry $\eta$ extends to a Hodge isometry $g$ of the trascendental lattices $T_X$ and $T_{X'}$. Indeed, consider the natural (non primitive) embeddings $i: T(X,\alpha) \hookrightarrow  T_X$ and $i': T(X',\alpha') \hookrightarrow T_{X'}$.  We set
$$H=\frac{T_X}{T(X,\alpha)} \quad \text{and} \quad H'=\frac{T_{X'}}{T(X',\alpha')},$$
which are cyclic subgroups of $\Z/d\Z$ of order $\kappa$. Thus, $H$ and $H'$ are the same subgroup, because they have the same order. If $\bar{\eta}$ denotes the automorphism of $\Z/d\Z$ induced by $\eta$, then
$$\bar{\eta}(H)=\text{id}(H)=H$$
by assumption. By \cite{Ni}, Proposition 1.4.2, we conclude that the isometry $\eta$ extends to an isometry $g: T_X \cong T_{X'}$. By construction, the isometry $g$ preserves the Hodge structures on $T_X$ and $T_{X'}$. 

Secondly, it is easy to check that, if the isomorphism $\bar{\eta}$ acts as the identity on $\Z/d\Z$, then also $\bar{g}$, induced by $g$, acts trivially on the discriminant groups.

Finally, we denote by $\Z l$ the rank one lattice which is the orthogonal complement of $T_X$ in $H^2(X,\Z) \cong \Lambda$. Since $d(T_X) \cong d(\Z l)$, by \cite{Ni}, Proposition 1.5.2 we conclude that the isometry $g$ extends to an isometry $f_{\Lambda}$ of $\Lambda$ and, therefore, the isometry $g$ extends to $f: H^2(X,\Z) \cong H^2(X',\Z)$. Furthermore, the restriction of $f_{\Lambda}$ to $\Z l$ is the identity, because by construction it induces the identity on the discriminant group of $\Z l$. In particular, we deduce that the isometry $f$ preserves the ample cones of $X$ and $X'$. By Torelli Theorem, there exists an isomorphism $F$ between the K3 surfaces $X'$ and $X$ such that $F^*=f$. Since, by definition, the isometry $f$ sends the class $\alpha$ to $\alpha'$, we conclude that $(X,\alpha)$ and $(X',\alpha')$ are isomorphic as twisted K3 surfaces, in contradiction with our assumption. Therefore, we conclude that $y$ and $y'$ are not the same point in $\mathcal{D}_d^{\text{mar}}$, as we wanted. 
\end{proof}

\begin{proof}[Proof of Theorem \ref{propmiatwist}]
The representatives of the $m'$ isomorphism classes of twisted Fourier-Mukai partners of order $\kappa$ of $(X,\alpha)$ determine $m'$ distinct period points $y_k \in \mathcal{D}^{\text{mar}}_d$ by Lemma \ref{lemmafondam}. Arguing as in the untwisted case, the proof follows from Proposition \ref{markedvslab}, Theorem 1.1 of \cite{Laza} and Remark \ref{rmk_BLMS+Huy}.  
\end{proof}

\begin{rmk}
\label{rmk_extwithBLM+twist}
As explained in Remark \ref{rmk_BLMS+Huy},  using \cite{BLM+} we can remove the assumption that $Y$ is very general special to Theorem \ref{propmiatwist}, which holds for every cubic fourfold with $H^{2,2}(Y,\Z)$ of rank $2$ in the fixed divisor $\mathcal{C}_d$.
\end{rmk}

\begin{rmk}
\label{rmkonperioddom}
Notice that it is necessary to assume that $\text{ord}(\alpha)=\text{ord}(\alpha')$, in order to extend the isometry $\eta$ to the trascendental lattices. Indeed, if this condition is not satisfied, then the discriminant groups of $T_X$ and $T_{X'}$ could not be isomorphic. Actually, we can prove that Lemma \ref{lemmafondam} does not hold in general without this assumption, by giving a counterexample in the untwisted case. 

We set $d=2 \cdot 13^2$, which is congruent to $2$ modulo $6$ and let $Y$ be a very general cubic fourfold in $\mathcal{C}_d$. Since $d$ satisfies condition \textbf{(a)}, there exists a K3 surface $X$, which is associated to $Y$. By the counting formula of Theorem \ref{exFMk3}, the K3 surface $X$ admits $2^0=1$ isomorphism class of Fourier-Mukai partners. On the other hand, by \cite{Ma}, Proposition 5.1, there exist $\varphi(13) \cdot 2^{-1}=6$ isomorphism classes of Fourier-Mukai partners of order $13$ of $X$. We denote by $(X',\alpha')$ one of them. Assume that there is a cubic fourfold $Y' \in \mathcal{C}_d$ such that $\tilde{H}(\A_{Y'},\Z) \cong \tilde{H}(X',\alpha',\Z)$. 
By Remark \ref{rmk_BLMS+Huy}, $Y'$ is a Fourier-Mukai partner of $Y$. On the other hand, by the counting formula of Theorem \ref{propmia}, every Fourier-Mukai partner of $Y$ is isomorphic to $Y$; it follows that $Y \cong Y'$. On the other hand, the K3 surfaces $X$ and $(X',\alpha')$ cannot clearly be isomorphic.  

This prevents us to have a well-defined map between $\mathcal{D}_d^{\text{mar}}$ and the period domain of generalized Calabi-Yau structures of hyperk\"ahler type (see \cite{Huyv} for the definition), and to generalize Theorem 5.3.2 and 5.3.3 of \cite{Hass} to the twisted case.
\end{rmk}

\subsection{Ma's formula in our setting}
The aim of this paragraph is to prove that if we consider a very general cubic fourfold $Y$ in $\mathcal{C}_d$ satisfying condition $\textbf{(a')}$ and $\textbf{(b)}$, then formula \eqref{CF} gives precisely the number of elements in the set $\FM^{r}(X,\alpha)$, where $(X,\alpha)$ is a twisted K3 surface associated to $Y$. The key point of the proof is the fact that the C\v{a}ld\v{a}raru Conjecture \ref{Caldconj} holds in this particular case.
  
\begin{prop}
\label{thmcasogen}
Let $(X,\alpha)$ be a twisted K3 surface such that there exist a special cubic fourfold $Y$ of discriminant $d$ and a Hodge isometry $\tilde{H}(X,\alpha,\Z) \cong \tilde{H}(\A_Y,\Z)$. If $X$ has $\text{rk}(H^{1,1}(X,\Z))=1$, and $9 \nmid d$, then the number of (isomorphism classes of) Fourier-Mukai partners of $(X,\alpha)$ of order $r$ is given by formula \eqref{CF}.
\end{prop}
\begin{proof}
Firstly, we observe that the C\v{a}ld\v{a}raru's Conjecture \ref{Caldconj} holds under our assumptions for every Fourier-Mukai partner $(X_1,\alpha_1)$ of $(X,\alpha)$. More precisely, we prove that if a K3 surface $X_1^{\prime}$ has the trascendental lattice $T_{X_1^{\prime}}$ Hodge isometric to $T_{X_1}$ via $g_1$, then the twisted K3 surface $(X_1^{\prime},\alpha_1^{\prime}:=g_1^{*}\alpha_1)$ is a Fourier-Mukai partner of $(X_1,\alpha_1)$. Indeed, the isometry $g_1$ restricts to the isometry of Hodge structures 
$$f:= (g_1)|_{T(X_1',\alpha_1')}: T(X_1^{\prime},\alpha_1^{\prime}) \cong T(X_1,\alpha_1).$$
Notice that there exists a Hodge isometry $T(X,\alpha) \cong T(X_1,\alpha_1)$; therefore, the discriminant group $d(T(X_1,\alpha_1))$ is cyclic. Thus, by \cite{Ni}, Theorem 1.14.4, the isometry $f$ extends to an isometry of Hodge structures
$$\phi_1: \tilde{H}(X_1^{\prime},\alpha_1^{\prime},\Z) \cong \tilde{H}(X_1,\alpha_1,\Z).$$ 
By \cite{Huy}, Lemma 2.3, we know that every Hodge structure on $\tilde{\Lambda}$ determined by a point in $\mathcal{D}'$ admits a Hodge isometry that reverses any given orientation of the four positive directions. As a consequence, up to composing with this isometry, we can assume that $\phi_1$ is orientation-preserving: by \cite{HuySt}, Theorem 0.1, we conclude that there exists an equivalence of categories $\D(X_1^{\prime},\alpha_1^{\prime}) \xrightarrow{\sim} \D(X_1,\alpha_1)$. In particular, we obtain that the map $\nu$ of \eqref{nu} is bijective. 

To conclude the proof, we show that the map $\mu$ of \eqref{mu} has image $\text{O}_{\text{Hdg}}(T(X,\alpha)) \setminus J^r(d(T(X,\alpha)))$; in particular, this implies that we have an equality in formula \eqref{CF}.

Let $x$ be in $J^r(d(T(X,\alpha)))$; by definition, $x$ is an element of $I^r(d(T(X,\alpha)))$ such that there exists an embedding
$$\varphi: U \rightarrow \tilde{M_x},$$ 
where 
$$\tilde{M_x}:=\langle \lambda(x), N(X,\alpha) \rangle \subset N(X,\alpha)^{\vee}$$
is an overlattice of $N(X,\alpha)$. By \cite{Ni}, Proposition 1.4.1, we have 
$$d(\tilde{M_x}) \cong \langle \lambda(x) \rangle^{\perp}/ \langle \lambda(x) \rangle \cong \langle x \rangle^{\perp}/ \langle x \rangle \cong d(T_x).$$
Thus, by \cite{Ni}, Proposition 1.6.1, we have an embedding $\tilde{M_x} \oplus T_x \hookrightarrow \tilde{\Lambda}$, with $\tilde{M_x}$ and $T_x$ both embedded primitively. We define the lattice
$$\Lambda_{\varphi}:= \varphi(U)^{\perp} \cap \tilde{\Lambda},$$ 
which is isometric to the K3 lattice $\Lambda$, with the Hodge structure induced from $T_x$. By the surjectivity of the period map, there exist a K3 surface $X_{\varphi}$ and a Hodge isometry
$$h: H^2(X_{\varphi},\Z) \cong \Lambda_{\varphi}.$$
We denote by $\alpha_{\varphi}$ the composition $\alpha_x \circ h|_{T_{X_{\varphi}}}$; then, we obtain a twisted K3 surface $(X_{\varphi}, \alpha_{\varphi})$. 

Now, we observe that the map $h$ induces the isometry
$$f: T(X_{\varphi},\alpha_{\varphi})= \ker \alpha_{\varphi} \cong \ker \alpha_x=T(X,\alpha).$$
Moreover, since $d(T(X,\alpha))$ is a cyclic group, applying \cite{Ni}, Theorem 1.14.4, we conclude that $f$ extends to a Hodge isometry
$$\tilde{f}: \tilde{H}(X_{\varphi},\alpha_{\varphi},\Z) \cong \tilde{H}(X,\alpha,\Z).$$
By \cite{Huy}, Lemma 2.3, we can assume that $\tilde{f}$ is orientation-preserving. By \cite{HuySt}, Theorem 0.1, we conclude that $(X_{\varphi},\alpha_{\varphi})$ belongs to $\text{FM}^r(X,\alpha)$. By construction, we have $\mu([(X_{\varphi},\alpha_{\varphi})])=[x]$. 

Finally, we observe that if $x$ and $x'$ in $J^r(d(T(X,\alpha)))$ are in the same orbit for the action of $\text{O}_{\text{Hdg}}(T(X,\alpha))$, then the twisted K3 surfaces $(X_{\varphi},\alpha_{\varphi})$ and $(X_{\varphi}^{\prime},\alpha_{\varphi}^{\prime})$, such that $\mu([(X_{\varphi},\alpha_{\varphi})])=[x]$ and $\mu([(X_{\varphi}^{\prime},\alpha_{\varphi}^{\prime})])=[x']$, are $\sim$-equivalent. Indeed, by hypothesis, there exists a Hodge isometry $\eta$ of $T(X,\alpha)$ which induces an isomorphism $\bar{\eta}$ on $d(T(X,\alpha))$ such that $\bar{\eta}(x)=x'$. Then, by \cite{Ni}, Proposition 1.4.2, the overlattices $\langle x,T(X,\alpha) \rangle \cong T_{X_{\varphi}}$ and $\langle x',T(X,\alpha) \rangle \cong T_{X_{\varphi}^{\prime}}$ are isomorphic. Moreover, this isomorphism sends $\alpha_{\varphi}$ to $\alpha_{\varphi}^{\prime}$, because it is an extension of $\eta$; this observation completes the proof of the proposition.
\end{proof}


\subsection{Application of Proposition \ref{thmcasogen}}

Let $Y$ be a very general special cubic fourfold of discriminant $d$ satisfying conditions $\textbf{(a')}$ and $\textbf{(b)}$. By Proposition \ref{thmcasogen}, the number of isomorphism classes of Fourier-Mukai partners of order $\kappa$ of $(X,\alpha)$ is
$$m'= \sum_x  \left\lbrace \sum_M \tau(x,M) + \varepsilon(r) \sum_{M'} \tau(x,M') \right\rbrace.$$
Let us write $m'$ in a more explicit way, in order to find numerical conditions on $d$ and $\kappa$, which guarantee the existence of non isomorphic Fourier-Mukai partners for $Y$. We consider only the case $\kappa \geq 2$, because we have already treated the untwisted case in Section 3. Let $c$ be the degree of the polarization class on $X$. Notice that $d=\kappa^2c$ (see \cite{Huy}, Lemma 2.13). 

\begin{lemma}
Let $g$ be a generator of the cyclic group $d(T(X,\alpha))$ of order $d$. Then
$$I^{\kappa}(d(T(X,\alpha)))= \lbrace (a \kappa c)g: a \in (\Z/\kappa \mathbb{Z})^{\times} \rbrace.$$
\end{lemma}
\begin{proof}
We observe that every element of the form $x= (a\kappa c)g$ with $a \in \left( \Z/ \kappa\mathbb{Z} \right)^{\times}$ belongs to $I^{\kappa}(d(T(X,\alpha)))$. Indeed, let $g$ be a generator of $d(T(X,\alpha))$ as in Proposition \ref{K_dperp}. An easy computation shows that $q_{T(X,\alpha)}((a \kappa c)g)  \in 2\mathbb{Z}$ and that $(a \kappa c)g$ has order $\kappa$. On the other hand, the elements of $I^{\kappa}(d(T(X,\alpha)))$ are all the possible generators of the unique subgroup of order $\kappa$ of $d(T(X,\alpha)) \cong \Z/d \mathbb{Z}$. 
\end{proof}

For every $x=(a \kappa c)g$ in $I^{\kappa}(d(T(X,\alpha)))$, we set 
$$\tilde{M}_x:=\langle \lambda(x),N(X,\alpha) \rangle \quad \text{and} \quad H_x:=\frac{\tilde{M}_x}{N(X,\alpha)}.$$
We point out that 
$$J^{\kappa}(d(T(X,\alpha)))=\lbrace x \in I^{\kappa}(d(T(X,\alpha))): \tilde{M}_x \cong U \oplus \Z l \text{ with } l^2=c \rbrace.$$
Indeed, given $x \in J^{\kappa}(d(T(X,\alpha)))$, let $(X_x,\alpha_x)$ be the twisted K3 surface such that $\mu([(X_x,\alpha_x)])=[x]$ (which exists because $\mu$ is surjective as showed in the proof of Proposition \ref{thmcasogen}). Then, by definition, we have 
$$N(X_x) \cong \langle \lambda(x),N(X,\alpha) \rangle \quad \text{and} \quad T_{X_x} \cong \langle x, T(X,\alpha) \rangle.$$
Since $T(X_x,\alpha_x) \cong T(X,\alpha)$, we have 
$$d=|d(T(X_x,\alpha_x))|= \text{ord}(\alpha_x)^2 |d(T_{X_x})|= \kappa^2 |d(T_{X_x})|,$$
which implies that
$$d(\tilde{M}_x) \cong d(T_{X_x}) \cong \Z/c \mathbb{Z}.$$
On the other hand, the opposite inclusion follows from the definition of $J^{\kappa}(d(T(X,\alpha)))$.

\begin{lemma}
Every element $x$ in $I^{\kappa}(d(T(X,\alpha)))$ belongs to $J^{\kappa}(d(T(X,\alpha)))$.
\end{lemma}
\begin{proof}
Let $\bar{x}=(\bar{a}\kappa c)g$ be the image via $\mu$ of the isomorphism class of the K3 surface $(X,\alpha)$, with $\bar{a}$ in $(\Z/\kappa \mathbb{Z})^{\times}$. By definition, we have 
$$U \oplus \Z l \cong N(X) \cong \langle \lambda(\bar{x}), N(X,\alpha) \rangle,$$
with $l^2=c$; in particular, the lattice $U \oplus \Z l$ is an overlattice of $N(X,\alpha)$. Let $x=(a \kappa c)g$ be an element in $I^{\kappa}(d(T(X,\alpha)))$. Since the groups $H_x$ and $H_{\bar{x}}$ are cyclic subgroups of $d(N(X,\alpha))$ of the same order, they are the same subgroup. By \cite{Ni}, Theorem 1.4.1, we conclude that the overlattices $U \oplus \Z l$ and $\tilde{M}_x$ are isomorphic. In particular, the element $x$ is in $J^{\kappa}(d(T(X,\alpha)))$. 
\end{proof}

\begin{prop}
\label{countm'}
We have 
$$
m' :=\# \emph{FM}^{\kappa}(X,\alpha)=
\begin{cases}
\varphi(\kappa) 2^{h-2} \quad & \text{if } \kappa>2 \text{ and } c=2 \\
\varphi(\kappa) 2^{h-1} \quad & \text{if } \kappa= 2 \text{ or } c>2,
\end{cases}
$$
where $h$ is the number of distinct prime factors in the prime factorization of $c/2$ if $c>2$, and $h=1$ if $c=2$.
\end{prop}
\begin{proof}
The previous lemmas and the fact that $\text{O}_{\text{Hdg}}(T(X,\alpha))= \lbrace \pm \text{id} \rbrace$ imply that
$$\#(\text{O}_{\text{Hdg}}(T(X,\alpha))) \setminus J^{\kappa}(d(T(X,\alpha)))=
\begin{cases}
1 \quad & \text{if } \kappa=2, \\
\frac{1}{2} \varphi(\kappa) \quad & \text{if } \kappa>2
\end{cases}
$$
where $\varphi$ denotes the Euler function. On the other hand, the only lattice $M_{\varphi}$ such that $\tilde{M}_x \cong U \oplus M_{\varphi}$ is $\Z l$ with $l^2=c$. Thus, our computation is actually the same used in \cite{Ma}, to prove Proposition 5.1. Indeed, we have 
$$\mathcal{G}(\Z l)= \lbrace \Z l \rbrace =
\begin{cases}
\mathcal{G}_1(\Z l) \quad \text{if } c=2, \\
\mathcal{G}_2(\Z l) \quad \text{if } c> 2.
\end{cases}
$$
Moreover, we notice that 
$$\text{O}(\Z l)=\lbrace \pm \text{id} \rbrace \quad \text{and} \quad \text{O}(d(\Z l))=
\begin{cases}
\lbrace \text{id} \rbrace \quad & \text{if } c= 2,\\
\left( \frac{\Z}{2 \mathbb{Z}} \right) ^h \quad & \text{if } c>2.
\end{cases}
$$
In particular, the order of the set $\text{O}(d(\Z l))$ is $2^h$ if $c>2$.
Finally, we observe that
$$\text{O}_{\text{Hdg}}(T_x,\alpha_x)=
\begin{cases}
\lbrace \pm \text{id} \rbrace \quad & \text{if } \kappa= 2, \\
\lbrace \text{id} \rbrace \quad & \text{if } \kappa>2.
\end{cases}
$$
So, if $\kappa>2$, then
$$
m'=
\begin{cases}
\frac{1}{2}\varphi(\kappa) \quad & \text{if } c= 2, \\
\frac{1}{2}\varphi(\kappa)2^{h}=\varphi(\kappa) 2^{h-1}  \quad & \text{if } c>2.
\end{cases}
$$
Otherwise, if $\kappa=2$, then
$$
m'=
\begin{cases}
1 \quad & \text{if } c = 2, \\
2^{h-1} \quad & \text{if } c>2,
\end{cases}
$$
as we claimed.
\end{proof}

By Proposition \ref{thmcasogen} and Proposition \ref{countm'}, the lower bound given by Theorem \ref{propmiatwist} is explicitely determined. In particular, it is easy to construct examples of very general twisted K3 surfaces and, consequently, of very general cubic fourfolds with an arbitrary large number of non-isomorphic Fourier-Mukai partners. 

\begin{ex}
Let us take $d=50$, which satisfies condition \eqref{condd} and $\textbf{(b)}$. A cubic fourfold in $\mathcal{C}_{50}$ has a twisted associated K3 surface with Brauer class of order $\kappa=5$. By Theorem \ref{propmiatwist} and Proposition \ref{countm'}, the very general element in $\mathcal{C}_{50}$ admits at least $\varphi(5)/2=4/2=2$ (isomorphim classes of) Fourier-Mukai partners. 
\end{ex}

Dipartimento di Matematica ``F.\ Enriques'', Universit\`a degli Studi di Milano, Via Cesare Saldini 50, 20133 Milano, Italy \\
\indent E-mail address: \texttt{laura.pertusi@unimi.it}\\
\indent URL: \texttt{http://www.mat.unimi.it/users/pertusi}

\end{document}